\documentclass[12pt]{article}
\def\ni{\noindent}
\usepackage{hyperref}
\usepackage{amssymb}
\usepackage{amsthm}
\usepackage[mathscr]{euscript}
\usepackage{setspace}
\usepackage{amssymb}
\usepackage{amsmath}
\usepackage[utf8]{inputenc}
\usepackage{amsmath}
\usepackage{mathtools} 
\usepackage{fourier}
\usepackage[left=3cm,right=3cm,top=2.5cm,bottom=2cm]{geometry}

\newtheorem{theorem}{Theorem}[section]
\newtheorem{definition}{Definition}[section]
\newtheorem{lemma}[theorem]{Lemma}
\newtheorem{corollary}{Corollary}[section]

\newtheorem{example}{Example}[section]
\numberwithin{equation}{section}

\begin{document}

\begin{center}{\bf \large
 Congruence properties on the parity of the numbers of $(a,b,m)$-copartitions of $n$}
\end{center}
\begin{center} \small {\bf Yudhisthira Jamudulia\footnote{Supported by World Bank, No. 751/GMU}}\\
Email: yjamudulia@gmuniversity.ac.in.
\end{center}

{\baselineskip .5cm \begin{center}
School of Mathematics \\
Gangadhar Meher University,
Amruta Vihar\\
Sambalpur - 768004,Odisha,INDIA
\end{center}}
\ni {\bf \small Abstract:}
We consider $cp_{a,b,m}(n)$, the number of $(a,b,m)$-copartitions of $n$. We find many infinitely many congruences modulo $2$ and $6$ for some particular value of $a$, $b$ and $m$.

\vspace{.5cm}
\ni {\bf \small 2000 Mathematics Subject Classification:}  11P83, 05A15, 05A17. \\
\ni {\bf \small Keywords:} Congruence, Dissection, Copartition.

\section{\large Introduction }
A partition of a positive integer $n$ is a non-increasing sequence of positive integers whose sum is n. If $p(n)$ denote the number of partition of  n, then the generating function of n is given by
\begin{equation*}
  \sum_{n=0}^{\infty} p(n) q^{n}= \frac{1}{(q;q)_{\infty}}
\end{equation*}
where, as customary, for any complex number a and $|q|<1$,
\begin{equation*}
  (a;q)_{\infty}=\prod_{n=1}^{\infty}(1-aq^{n-1}).
\end{equation*}
Andrews \cite{a1} introduced the function $\mathscr{EO}^{*}(n)$, which counts the number of integer partitions of $n$ with all even parts smaller than all odd parts, where the only part appearing an odd number of times is the largest even part. He also studied its many interesting properties including the generating function, simply, $\frac{1}{2}(\nu(q)+\nu(-q))$, where
$$\nu(q)=\sum_{n=0}^{\infty}\frac{q^{n^2+n}}{(-q;q^2)_{n+1}} $$
is of Watson's third order mock theta function.
Chern \cite{c1}, provided a combinatorial proof of the generating function for $\mathscr{EO}^{*}(n)$ and studied several further properties in \cite{c2}.

Burson and Eichhorn in \cite{b1} generalized $\mathscr{EO}^{*}(n)$ by introducing new partition-theoretic objects called copartitions, which reveal an inherent symmetry in partitions counted by $\mathscr{EO}^{*}(n)$ that was not previously obvious. Copartitions are counted by the function $cp_{a,b,m}(n)$, where $cp_{1,1,2}(n)=\mathscr{EO}^{*}(2n)$.
\begin{definition}
An $(a,b,m)$-copartition is a triple of partition $(\gamma, \rho, \sigma)$, where each of the parts of $\gamma$ is at least $a$ and congruent to $a\,\, (mod \,m)$, each of the parts of $\gamma$ is at least $b$ and congruent to $b\,\, (mod \, m)$, and $\rho$ has the same number of parts $\sigma$, each of which have size equal to $m$ times the number of parts $\gamma$.
\end{definition}
\begin{example}
The $(1, 3 ,4)$-copartitions of size $12$ are
$$(\{9, 1^3\},{\O},{\O}), (\{5^2,1^2\},{\O},{\O}),(\{5,1^7\},{\O},{\O}),(\{1^{12}\},{\O},{\O}) $$
$$(\{5\},\{4\},\{3\}), (\{1\},\{4\},\{7\}), and \,\, ({\O},{\O},\{3^4\}) $$
Thus $cp_{1,3,4}(12)=7$.
\end{example}

\begin{theorem}\cite[Theorem 3.5]{b1}
$cp_{1,1,2}(n)=\mathscr{EO}^{*}(2n)$.
\end{theorem}
\ni Now the generating function for the copartitions can be given as follows:
\begin{theorem}\cite[Theorem 3.8]{b1}
Define $cp_{a,b,m}(w,s,n)$ to be the number of $(a,b,m)$-copartitions of size $n$ that have $w$ ground parts and $s$ sky parts. Then,
\begin{equation}\label{1.0}
  cp_{a,b,m}(x,y,q) = \sum_{n=0}^{\infty}\sum_{w=0}^{\infty}\sum_{s=0}^{\infty}cp_{a,b,m}x^sy^wq^n
   = \frac{(xyq^{a+b};q^m)_{\infty}}{(xq^{b};q^m)_{\infty}(yq^{a};q^m)_{\infty}}
\end{equation}
\end{theorem}

\ni The following special cases has been discussed.
\begin{theorem}\cite[Theorem 4.5]{b1}
\begin{equation}
  cp_{1,1,1}(n)=\sum_{k=0}^{n}p(k).
\end{equation}
\end{theorem}

\begin{theorem}\cite[Theorem 4.10]{b1}
\begin{equation}
  cp_{0,1,1}(n)=\sum_{k=0}^{n-1}p(k)d(n-k).
\end{equation}
\ni where $d(n)$ is the number of divisions of $n$.
\end{theorem}

\begin{theorem}\cite[Theorem 4.16]{b1}
\begin{equation}
  cp_{0,0,1}(n)=-p(n)+2\sum_{k=0}^{n-1}p(k)d(n-k).
\end{equation}
\ni where $d(n)$ is the number of divisions of $n$.
\end{theorem}

\ni Burson and Eichhorn \cite{b2} also studied the parity of $cp_{3,1,4}(n)$ and $cp_{5,1,6}(n)$.
\begin{corollary} \cite[Corollary 4.5]{b2}
\ni For any prime $p>3$, $p\cong 3 \,\, (mod \, \,4)$, let $ 24 \delta \cong 1 \,\, (mod \,\,p^2)$. Then
\begin{equation}
cp_{3,1,4}(p^2k+pt-5\delta) \cong 0 \,\, (mod \,\, 2)\label{c26}
\end{equation}
\ni for $t=1, 2, \cdots p-1$ and every nonnegative integer $k$.
\end{corollary}

\ni Some special cases of \eqref{c26} are as follows:
\begin{corollary}\cite[Corollary 4.6]{b2}
For $r=3, 17, 24, 31, 38, 45$, we have
\begin{equation}
cp_{3,1,4}(49k+r) \cong 0 \,\,(mod \,\, 2)
\end{equation}
for every nonnegative integer $k$.
\end{corollary}

\begin{corollary}\cite[Corollary 4.7]{b2}
For $r=3, 14, 36, 47, 58, 69, 80, 91, 102, 113$, we have
\begin{equation}
cp_{3,1,4}(49k+r) \cong 0 \,\, (mod \,\, 2)
\end{equation}
for every nonnegative integer $k$.
\end{corollary}

\begin{corollary}\cite[Corollary 4.10]{b2}
For any prime $p>2$, $p\cong 2 \,\, (mod \, 3)$, let $ 6 \delta \cong 1 \, (mod \,\,p^2)$. Then
\begin{equation}
cp_{5,1,6}(p^2k+pt-\delta)\, \cong 0 \,\, (mod \,\, 2)\label{c27}
\end{equation}
for $t=1, 2, \cdots p-1$ and every nonnegative integer $k$.
\end{corollary}

\ni Some special cases of \eqref{c27}  are as follows:
\begin{corollary}\cite[Corollary 4.11]{b2}
For $r=9, 14, 19, 24$, we have
\begin{equation}
cp_{5,1,6}(25k+r) \cong 0\,\,(mod \,\, 2)
\end{equation}
for every nonnegative integer $k$.
\end{corollary}

\begin{corollary}\cite[Corollary 4.12]{b2}
For $r=9, 31, 42, 53, 64, 75, 86, 97, 108, 119$, we have
\begin{equation}
cp_{5,1,6}(121k+r) \cong  0 \,\,(mod \,\, 2)
\end{equation}
for every nonnegative integer $k$.
\end{corollary}

In this paper, we establish many infinite family of congruences properties of $cp_{3, 1, 4}(n)$ and $cp_{5,1,6}$ modulo $2$ and $6$.

\section{Preliminary Results}
In this section, we list few dissection formulas which are useful in proving our main results. As customary, we use
$$f_k= f(-q^k)=\prod_{n=1}^{\infty}(1-q^{nk})=(q^k;q^k)_{\infty}, \,\, for\,\, k\,\, \geq 1.$$.
\begin{lemma}\cite[Theorem 2.2]{cui}
For any prime $p \geq 5$,
\begin{equation}
  f(-q)\,=\, \sum_{\substack{k=-\frac{p-1}{2} \\ k\neq \frac{\pm p-1}{6}}}^{\frac{p-1}{2}} (-1)^{k} q^{\frac{3k^2+k}{2}}f(-q^{\frac{3p^2-(6k+1)p}{2}},-q^{\frac{3p^2+(6k+1)p}{2}})
  +(-1)^{ \frac{\pm p-1}{6}} q^{\frac{p^2-1}{24}}f(-q^{p^2}).\label{c6}
\end{equation}
where $\pm$ depends on the conditions that $(\pm p-1)/6$ should be an integer. Moreover, note that $(3k^2+k)/2 \ncong (p^2-1)/24 ( \mod p)$ as k runs through the range of the summation.
\end{lemma}

\begin{lemma} For positive integers $k$ and $m$, we have
  \begin{eqnarray}
    f_{2k}^{m} &\cong & f_{k}^{2m} (mod \,\,2) \label{c2} \\
    f_{3k}^{m} &\cong & f_{k}^{3m} (mod\,\, 3) \label{c21}
  \end{eqnarray}
\end{lemma}

\section{Main Results}
In this section, we state and prove our main results.

\begin{theorem} For any prime $p \geq 5$ and $\alpha, n \geq 0$, we have
\begin{equation}
\sum_{n=0}^{\infty}cp_{3,1,4}\left( p^{2 \alpha}n+5 \frac{p^{2 \alpha}-1}{24} \right) q^n
 \cong f^{5}(-q)\,\, (mod\,\, 2) \label{c10}.
 \end{equation}
 \end{theorem}
\begin{proof}
Setting $x=y=1$, $a=3$, $b=1$ and $m=4$ in \eqref{1.0}, we have

\begin{equation}\label{c1}
   \sum_{n=0}^{\infty}cp_{3,1,4}q^n
   = \frac{(q^{4};q^4)_{\infty}}{(q;q^4)_{\infty}(q^{3};q^4)_{\infty}}=\frac{(q^4;q^4)_{\infty}}{(q:q^2)_{\infty}}=\frac{f_4 f_2}{f_1}.
\end{equation}
Applying \eqref{c2} in \eqref{c1}, we obtain
\begin{equation}\label{c3}
\sum_{n=0}^{\infty}cp_{3,1,4}(n) q^n \cong \,\, f_{1}^{5}\,\, (mod\,\, 2).
\end{equation}
Applying \eqref{c6}\\
\ni $\sum_{n=0}^{\infty} cp_{3,1,4}\left(p^{2 \alpha}n+5\frac{p^{2\alpha}-1}{24}\right) q^n$
\begin{align}\label{c4}
  = & \left(\sum_{\substack{k=-\frac{p-1}{2} \\ k\neq \pm\frac{p-1}{6}}}^{\frac{p-1}{2}} (-1)^{k} q^{\frac{3k^2+k}{2}}f(-q^{\frac{3p^2-(6k+1)p}{2}},-q^{\frac{3p^2+(6k+1)p}{2}})
  +(-1)^{\pm \frac{p-1}{6}} q^{\frac{p^2-1}{24}}f(-q^{p^2}) \right)^5 (mod \,\, 2)
\end{align}
Extracting the term containing $q^{pn+5\frac{p^2-1}{24}}$from both sides of \eqref{c4}, and replacing $q^p$ by $q$, we obtain
\begin{equation}\label{c8}
\sum_{n=0}^{\infty} cp_{3,1,4}\left(p^{2\alpha+1}n+5\frac{p^{2\alpha+1}-1}{24}\right)q^n \cong \,\,f^{5}(-q^p) \,\,(mod \,\,2).
\end{equation}
Again extracting the term containing $q^{pn}$ from both sides of \eqref{c8} and replacing $q^{p}$ by $q$, we arrive at

\begin{equation}\label{c9}
\sum_{n=0}^{\infty} cp_{3,1,4}\left(p^{2\alpha+2}n+5\frac{p^{2\alpha+2}-1}{24}\right)q^n \cong \,\,f^{5}(-q)\,\, (mod \,\,2).
\end{equation}
which is the $\alpha +1$ term of \eqref{c10}.
\end{proof}
\begin{corollary}
For $p \geq 5$, $\alpha \geq 1$ and $n \geq 0$, we have
\begin{equation}\label{c12}
  \sum_{n=0}^{\infty} cp_{3,1,4}\left(p^{2\alpha}n+\frac{(24j+5p)p^{2\alpha-1}-5}{24} \right) \cong 0 \,\,(mod \,\,2).
\end{equation}
where $ j=1,2,\cdots  p-1$.
\end{corollary}
\begin{proof}
 Comparing the coefficients of $q^{pn+j}$, $1 \leq j \leq p-1$  in \eqref{c8}, we easily obtain \eqref{c12}.
\end{proof}
\begin{corollary}
  For $p \geq 5$, $\alpha \geq 0$ and $ n \geq 0$, we have
  \begin{equation}
  cp_{3,1,4}\left(p^{2\alpha +1}n+\frac{(24j+5)p^{2\alpha }-5}{24} \right) \cong 0\,\, (mod\,\, 2), \label{c13}
  \end{equation}
  for $j=1,2,\cdots p-1$ and $\left(\frac{24j+5}{p}\right)=-1.$
\end{corollary}
\begin{proof}
According to \eqref{c6} and Theorem \eqref{c10}, for any integer $j$ with $0 \leq j \leq p-1$, if $ j \ncong (3k^2+k)/2 \,\,(mod\,\, p)$ for $|k| \leq (p-1)/2$, then we have
$$cp_{3,1,4}\left( p^{2 \alpha}(pn+j)+5\frac{p^{2\alpha}-1}{24}\right) \cong 0 \,\,(mod \,\,2)$$ which gives \eqref{c13}.
\end{proof}
\begin{theorem}
For the primes $ p_1, p_2, p_3 \cdots p_l$, $l \geq 0 $ and $n \geq 0$, we have
\begin{equation}
  \sum_{n=0}^{\infty}cp_{3,1,4}\left(\prod_{s=1}^{l}p_{s}^{2}+5 \left( \frac{\prod_{s=1}^{l}p_s^{2}-1}{24}\right) \right) q^n \cong f^{5}(-q) \,\, (mod \,\,2)\label{c14}
\end{equation}
\end{theorem}
\begin{proof}
  Proof can be completed by induction on $l$ where the initial case is \eqref{c10}. Assume that \eqref{c14} is true for $l$. Then based on \eqref{c6} for prime $p_{l+1}$, we have
\begin{align}\label{c15}
  \sum_{n=0}^{\infty} cp_{3,1,4}\left( \prod_{s=1}^{l}p_{s}^{2}\left(p_{l+1}^{2}n+5\frac{p_{l+1}^{2}-1}{24} \right)+5\left(\frac{\prod_{s=1}^{l}p_{s}^{2}-1}{24}\right)\right)q^n \nonumber &  \\
 =\sum_{n=0}^{\infty}cp_{3,1,4} \left(\prod_{s=1}^{l+1}p_{s}^{2}n+5\left(\frac{\prod_{s=1}^{l+1}p_{s}^{2}-1}{24}\right) \right)q^n \cong \,\, f^{5}(-q) \,\, (mod \,\, 2) \nonumber &
\end{align}
which is the case of $l+1$.
\end{proof}

\begin{theorem} For any prime $p \geq 5$ and $\alpha$, $n \geq 0$, we have
  \begin{equation}\sum_{n=0}^{\infty} cp_{5,1,6}\left(p^{2\alpha}n+4 \frac{p^{2\alpha}-1}{24}\right) \cong f^{4}(-q)\,\, (mod \,\, 6). \label{c16}
  \end{equation}
\end{theorem}
\begin{proof}
Setting $x=y=1$, $a=5$, $b=1$ and $m=6$ in \eqref{1.0}, we have

\begin{equation}\label{c17}
   \sum_{n=0}^{\infty}cp_{5,1,6}q^n
   = \frac{(q^{6};q^6)_{\infty}}{(q;q^6)_{\infty}(q^5;q^6)_{\infty}}=\frac{(q^6;q^6)_{\infty}(q^3;q^6)_{\infty}}{(q:q^2)_{\infty}}=\frac{f_2 f_3}{f_1}.
\end{equation}
Applying \eqref{c2} and \eqref{c21} in \eqref{c17}, we obtain
\begin{equation}\label{c22}
\sum_{n=0}^{\infty}cp_{5,1,6}(n) q^n \cong \,\, f_{1}^{4} \,\,(mod\,\, 6).
\end{equation}
Applying \eqref{c6}\\
\ni $\sum_{n=0}^{\infty} cp_{5,1,6}\left(p^{2 \alpha}n+4\frac{p^{2\alpha}-1}{24}\right) q^n $
\begin{align}\label{c18}
& \left(\sum_{\substack{k=-\frac{p-1}{2} \\ k\neq \pm\frac{p-1}{6}}}^{\frac{p-1}{2}} (-1)^{k} q^{\frac{3k^2+k}{2}}f(-q^{\frac{3p^2-(6k+1)p}{2}},-q^{\frac{3p^2+(6k+1)p}{2}})
  +(-1)^{\pm \frac{p-1}{6}} q^{\frac{p^2-1}{24}}f(-q^{p^2}) \right)^4 (mod \,\, 6)
\end{align}
Extracting the term containing $q^{pn+4\frac{p^2-1}{24}}$from both sides of \eqref{c18}, and replacing $q^p$ by $q$, we obtain
\begin{equation}\label{c19}
\sum_{n=0}^{\infty} cp_{5,1,6}\left(p^{2\alpha+1}n+4\frac{p^{2\alpha+1}-1}{24}\right)q^n \cong\,\, f^{4}(-q^p) \,\,(mod \,\,6).
\end{equation}
Again extracting the term containing $q^{pn}$ from both sides of \eqref{c19} and replacing $q^{p}$ by $q$, we arrive at

\begin{equation}\label{c20}
\sum_{n=0}^{\infty} cp_{5,1,6}\left(p^{2\alpha+2}n+4\frac{p^{2\alpha+2}-1}{24}\right)q^n \cong\,\, f^{4}(-q)\,\, (mod \,\,6).
\end{equation}
which is the $\alpha +1$ term of \eqref{c16}.
\end{proof}

\begin{corollary}
For $p \geq 5$, $\alpha \geq 1$ and $n \geq 0$, we have
\begin{equation}\label{c23}
  \sum_{n=0}^{\infty} cp_{5,1,6}\left(p^{2\alpha}n+\frac{(24j+4p)p^{2\alpha-1}-4}{24} \right) \cong 0 \,\,(mod \,\,6).
\end{equation}
where $ j=1,2,\cdots  p-1$.
\end{corollary}
\begin{proof}
 Comparing the coefficients of $q^{pn+j}$, $1 \leq j \leq p-1$  in \eqref{c16}, we easily obtain \eqref{c23}.
\end{proof}
\begin{corollary}
  For $p \geq 5$, $\alpha \geq 0$ and $ n \geq 0$, we have
  \begin{equation}\label{c24x}
  cp_{5,1,6}\left(p^{2\alpha +1}n+\frac{(24j+4)p^{2\alpha }-4}{24} \right) \cong 0\,\, (mod\,\, 6),
  \end{equation}
  for $j=1,2,\cdots p-1$ and $\left(\frac{24j+4}{p}\right)=-1.$
\end{corollary}
\begin{proof}
According to \eqref{c6} and Theorem \eqref{c16}, for any integer $j$ with $0 \leq j \leq p-1$, if $ j \ncong (3k^2+k)/2 \,\,(mod\,\, p)$ for $|k| \leq (p-1)/2$, then we have
$$cp_{5,1,6}\left( p^{2 \alpha}(pn+j)+4\frac{p^{2\alpha}-1}{24}\right) \cong 0 \,\,(mod \,\,6)$$ which gives \eqref{c24x}.
\end{proof}

\begin{theorem}
For the primes $ p_1, p_2, p_3 \cdots p_l$, $l \geq 0 $ and $n \geq 0$, we have
\begin{equation}
  \sum_{n=0}^{\infty}cp_{5,1,6}\left(\prod_{s=1}^{l}p_{s}^{2}+4 \left( \frac{\prod_{s=1}^{l}p_s^{2}-1}{24}\right) \right) q^n \cong f^{4}(-q) \,\, (mod \,\,6)\label{c24}
\end{equation}
\end{theorem}
\begin{proof}
  Proof can be completed by induction on $l$ where the initial case is \eqref{c16}. Assume that \eqref{c24} is true for $l$. Then based on \eqref{c6} for prime $p_{l+1}$, we have
\begin{align}\label{c15}
  \sum_{n=0}^{\infty} cp_{5,1,6}\left( \prod_{s=1}^{l}p_{s}^{2}\left(p_{l+1}^{2}n+4\frac{p_{l+1}^{2}-1}{24} \right)+4\left(\frac{\prod_{s=1}^{l}p_{s}^{2}-1}{24}\right)\right)q^n \nonumber &  \\
 =\sum_{n=0}^{\infty}cp_{5,1,6} \left(\prod_{s=1}^{l+1}p_{s}^{2}n+4\left(\frac{\prod_{s=1}^{l+1}p_{s}^{2}-1}{24}\right) \right)q^n \cong \,\, f^{4}(-q) \,\, (mod \,\, 6) \nonumber &
\end{align}
which is the case of $l+1$.
\end{proof}
\begin{center}{\bf Concluding Remarks}
\end{center}
The author has found infinite family of congruences of parity of  $(a, b, m)$ -copartition modulo 2 and 6. In future, the author in search of other possible combination of $a$, $b$, $m$  for different modulus. By the application of different dissection of theta functions, one can obtain other family of congruences for different modulus. The author is also in search of new congruence results of the same partition function by the application of tools of modular form.

\begin{center}{\bf Acknowledgement}
\end{center}
The author is thankful to World Bank, for awarding research project $[No.751/GMU]$ under which this work has been done.

\section*{\begin{center}Author Information\end{center}}

\ni Yudhisthira Jamudulia,
\ni School of Mathematics,
Gangadhar Meher Uniersity,\\
Amruta Vihar, Sambalpur,
Odisha-768004,India,
yjamudulia@gmuniversity.ac.in

\begin{thebibliography}{99}
\bibitem{a1}
G.E. Andrews, {\em Integer partitions with even parts below odd parts and the mock theta functions}, Ann. Comb., $22(2018), pp.433-445$.
\bibitem{berndt1}
B. C. Berndt, {\em Ramanujan's notebook, Part-III}, Springer-Verlag, New york, $1991$
\bibitem{b1}
Hannah E. Burson and Dennis Eichhorn,{\em Copatitions}, Ann. Comb. $(2022)$. $https://link.springer.com/article/10.1007/s00026-022-00607-1$, $https://arxiv.org/pdf/2111.04171.pdf$ $(2021)$.
\bibitem{b2}
Hannah E. Burson and Dennis Eichhorn, {\em On the parity of the number of $(a,b,m)$-copartitions of $n$}, $https://arxiv.org/abs/2201.04247$ $(2022)$.
\bibitem{c1}
S. Chern, {\em On a problem of George Andrews concerning partitions with even parts below odd parts}, Afr. Mat.,$\textbf{30}(2019), pp.691-695$.
\bibitem{c2}
S. Chern, {\em Note on partitions with even parts below odd parts}, Math. Notes, $\textbf{110}(2021), pp.454-457$.
\bibitem{cui}
S.-P Cui, N.S.S. Gu, {\em Arithmetic properties of $l$-regular partitions}, Adv. Appl. Math. $\textbf{51}(2013), pp.507-523$.
\end{thebibliography}
\end{document}